\documentclass[12pt,a4paper,oneside]{article}
\usepackage[utf8]{inputenc}
\usepackage{amssymb}
\usepackage{amsmath,amsthm,amsfonts}
\usepackage[T1]{fontenc}
\usepackage[margin=1in]{geometry}
\usepackage{indentfirst}
\theoremstyle{theorem}
\newtheorem{theorem}{{\sc Theorem}}[section]
\newtheorem{cor}[theorem]{{\sc Corollary}}

\theoremstyle{definition}
\newtheorem{ex}[theorem]{{\sc Example}}
\newtheorem{rem}[theorem]{{\sc Remark}}
\title{A note about irreducibility of a resultant}
\author{Beata Hejmej\footnote{e-mail: bhejmej1f@gmail.com, Department of Mathematics, Pedagogical University of Cracow, Podchor\k{a}\.zych~2, PL-30-084 Kraków, POLAND}}

\begin{document}
\maketitle
\begin{abstract}
We present a theorem about irreducibility of  a polynomial that is the  resultant of two others polynomials.~The~proof of this fact is based on the field theory.~We also consider the converse theorem and some examples.
\end{abstract}
\section{Preliminaries}
At the beginning,  we recall some basic definitions and facts from the field theory. 

Every nonzero homomorphism of fields is called an {\it embedding}.~For a field extension $F<E$ and an~embedding $\sigma\colon F \hookrightarrow L$, an embedding $\bar{\sigma}\colon E \hookrightarrow L$ such that $\bar{\sigma}|_F=\sigma$ is called an~{\it extension} of $\sigma$. An extension of the identity map $F \hookrightarrow F<L$ is called an $F$--{\it embedding}. 

If $\sigma \colon F \hookrightarrow E$ is an embedding and $f=a_n X^n+\cdots +a_0 \in F[X]$, then we set  $f^{\sigma} := \sigma(a_n) X^n + \cdots +\sigma(a_0) \in E[X]$.

We say that $E$ is a {\it splitting field over} $F$ of a family $\mathcal{F}\subset F[X]$  if  every polynomial $f\in \mathcal{F}$ splits over $E$  and $E=F(S)$, where $S$ is the set of all roots of polynomials from the family $\mathcal{F}$ (we assume that $S\subset \bar{F}$, the algebraic closure of $F$).

An algebraic extension $F<E$, where $E<\bar{F}$, is said to be  {\it normal}  if $E$ is a splitting field of some family $\mathcal{F}\subset F[X]$. In this situation we also say that $E$ is {\it normal over} $F$.

Consider a field  extension $F<E$. The set ${\rm Gal}(E/F)$ of all $F$--automorphisms of $E$  is a~group under the composition of mappings, which we call the {\it Galois group of the~extension} $F<E$.

Let $f$ be a  polynomial over a field $F$. We define the {\it Galois group of the polynomial} $f$ as the  Galois group of the extension $F<L_f$, where $L_f$ is the splitting field of $f$. We denote this group by ${\rm Gal}(f)$. It acts on the set $Z_f$ of all roots of $f$ by an obvious way.

The following theorem collects some well known properties of extension of fields, all of which can be found in~\cite{roman}.

\begin{theorem}
Assume that $F<E$ is an algebraic field extension and  $L$  is an algebraically closed field.~Then:
\begin{itemize}

\item[{\rm (i)}] Let $\alpha \in E$ and  $m_{\alpha,F}$ be the minimal polynomial of $\alpha$ over $F$. If $\sigma\colon F \hookrightarrow L$ is an embedding and  $\beta \in L$ is a root of  $m_{\alpha, F}^{\sigma}$, then $\sigma$ can be extended to an embedding $\bar{\sigma}\colon E \hookrightarrow L$ such that $\bar{\sigma}(\alpha)=\beta$. 

\item[{\rm (ii)}] If  $E<\bar{F}$, then $F<E$ is normal if and only  if  every $F$--embedding $E \hookrightarrow \bar{F}$ is an automorphism of $E$. 
\end{itemize}\label{aa}
\end{theorem}

We need a slight generalization of a well known property of the Galois group. 

\begin{theorem}
Let $f$ be a monic polynomial over a field $F$, $\deg f >0$.~Then ${\rm Gal}(f)$ acts transitively on the set of all roots of the polynomial $f$ if and only if $f$ is a power of  some monic irreducible polynomial. \label{sep}
\end{theorem}
\begin{proof}
Assume that ${\rm Gal}(f)$ acts transitively on the set $Z_f$ of all roots of the polynomial $f$.~Let $f_1,\dots, f_s\in F[X]$ be all distinct irreducible factors of $f$  (all of them are monic).~Take $r_i\in Z_{f_i}$, $r_j\in Z_{f_j}$.~Then $r_i,r_j\in Z_f$ and according to our assumption there exists an automorphism $\sigma \in {\rm Gal}(f)$ such that $\sigma(r_i)=r_j$.~Thus $0=\sigma(f_i(r_i))=f_i(r_j)$.~It follows that $f_j|f_i$, so $f_i=f_j$.~This implies that $f$ is a power of a monic irreducible polynomial.

Conversely, assume that $f$ is a power of a monic irreducible polynomial $g\in F[X]$.~Then $Z_f=Z_g$ and $g$ is the minimal polynomial of every element of $Z_f$.~Take $r_i,r_j\in Z_f$.~Since the extension $F<L_f$  is algebraic, the identity $F\hookrightarrow \overline{F}$ can be extended to an $F$--embedding $\sigma\colon L_f\hookrightarrow \overline{F}$ such that $\sigma(r_i)=r_j$ (Theorem \ref{aa}(i)).~According to the normality of the extension $F<L_f$, the $F$--embedding $\sigma$ must be an element of the group ${\rm Gal}(f)$ (Theorem \ref{aa}(ii)).~Therefore ${\rm Gal}(f)$ acts transitively on the set $Z_f$.
\end{proof}
\section{Main theorem}
Let $k$ be a field and ${\rm Res}_Y(f,g)$ denote  the  resultant of polynomials $f,g\in k[Y,T]$ with respect to the variable $Y$. 

\begin{theorem}
Let $f, g \in k[Y]$ be monic.~If $g$ is irreducible in the ring $k[Y]$ then the~polynomial 
$h=(-1)^{\deg g}\textnormal{Res}_Y(g,f-T)\in k[T]$  is a power of some irreducible polynomial.\label{main}
\end{theorem}
\begin{proof} 
Let $Z_g=\{y_1,\dots,y_m\}$.~Observe that  $h=\prod_{i=1}^m(T-f(y_i))$, so  $Z_h=\{f(y_i): i=1,\dots,m\}$.~Set $L_g:=k(y_1,\dots,y_m)$ and $L_h:=k(f(y_1),\dots,f(y_m))$.~It is obvious that $k\subset L_h\subset L_g\subset \bar{k}$.~Take $i,j\in\{1,\dots,m\}$.~Since the polynomial $g$ is irreducible,  Theorem~\ref{sep} implies that the action of ${\rm Gal}(g)$ on the set $Z_g$ is transitive.~It follows that $\sigma(y_i)=y_j$ for some $\sigma \in {\rm Gal}(g)$.~Therefore $\sigma|_{L_h}\colon L_h\hookrightarrow \overline{k}$ is a $k$--embedding.~The extension $k<L_h$ is normal, so according to Theorem \ref{aa}(ii), we have that $\sigma|_{L_h}$ is a~$k$--automorphism of $L_h$.~Thus $\tau:=\sigma|_{L_h}\in  {\rm Gal}(h)$ and $\tau(f(y_i))=\sigma(f(y_i))=f(\sigma(y_i))=f(y_j)$.~It means that ${\rm Gal}(h)$ acts transitively on the set $Z_h$ and by Theorem \ref{sep}   the statement follows. 
\end{proof}

Now, we present some examples connected with the converse theorem.

The first example shows that, in general, the converse to Theorem \ref{main} does not hold.
\begin{ex}
Let $f=Y^2-X^3\in \mathbb{C}((X))[Y]$ and $g=(Y^2-X^3)^2-X^7\in\mathbb{C}((X))[Y]$. Then $h=(T^2-X^7)^2\in\mathbb{C}((X))[T]$ is the square of the irreducible polynomial, but $g$ has two irreducible factors in $\mathbb{C}((X))[Y]$ (see \cite{abh}). (Here $\mathbb{C}((X))$ denotes the quotient field of the ring $\mathbb{C}[[X]]$ of formal power series.)
\end{ex}

 If  we assume that $h$ is irreducible, then the converse to Theorem \ref{main} holds. 
 
 \begin{cor}
 Let $f,g\in k[Y]$ be monic.~If $h=(-1)^{\deg g}\textnormal{Res}_Y(g,f-T)\in k[T]$ is irreducible, then $g$ is also irreducible.\label{cor}
 \end{cor}
 \begin{proof}
Assume that  $g=g_1 \cdots g_s$, where $k>1$ and $g_1, \dots, g_s\in k[Y]$ are monic and irreducible. Then $$h=(-1)^{\deg g}{\rm Res}_Y(g_1,f-T)\cdots {\rm Res}_Y(g_s,f-T).$$ Since $g_1,\dots,g_s$ are monic and irreducible over $k$, Theorem \ref{main} implies that each \linebreak ${\rm Res}_Y(g_i,f-T)$ is a power of some irreducible polynomial.~This means that $h$ is reducible in $k[T]$.
 \end{proof}
 Consider the following example.
 \begin{ex}
 Let $f=Y^2-X^3$ and $g=(Y^2-X^3)^2-X^5Y$ be polynomials over the field $\mathbb{C}((X))$. Let $w(i,j):=4i+13j$ be a weight. Then the initial quasi-homogeneous part of $h=T^4-X^{10}T-X^{13}\in \mathbb{C}[[X,T]]$ is equal to $T^4-X^{13}$. Since the integers $4$ and $13$ are coprime, the polynomial $T^4-X^{13}$ is irreducible in the ring $\mathbb{C}[X,T]$. Therefore  Hensel's Lemma (see \cite[Lemma A1]{luengo}) implies that   $h$ is irreducible in the ring $\mathbb{C}((X))[T]$.~By Corollary \ref{cor} the polynomial $g$ is irreducible over $\mathbb{C}((X))$. 
 \end{ex}
 \begin{rem}
 Polynomials $g_1=(Y^2-X^3)^2-X^7$ and $g_2=(Y^2-X^3)^2-X^5Y$ are taken from \cite{abh}. Both were proposed by Tzee-Char Kuo.
 \end{rem}
 
 {\bf Acknowledgements.}~The author would like to thank Professors Evelia Rosa García Barroso, Janusz Gwoździewicz and Kamil Rusek for  useful comments and helpful suggestions concerning this paper.

\end{document}